\documentclass{amsart}
\usepackage[utf8]{inputenc}
\usepackage{color}

\usepackage[margin=1.0in]{geometry}

\usepackage{amsmath}
\usepackage{amsthm}
\usepackage{amssymb}
\usepackage[matrix,arrow,curve]{xy}
\usepackage{enumitem} 
\usepackage{graphicx}
\usepackage{mdwlist}	
\usepackage{mathrsfs}
\usepackage{tikz}
\usetikzlibrary{patterns}
\usepackage{caption}
\usepackage{subcaption}
\usepackage{makecell}
\usepackage{ textcomp }

\newcommand\setItemnumber[1]{\setcounter{enum\romannumeral\@enumdepth}{\numexpr#1-1\relax}}

\newcommand{\p}{\mathbb{P}}
\DeclareMathOperator{\Bir}{Bir}
\DeclareMathOperator{\Aut}{Aut}

\DeclareMathOperator{\PGL}{PGL}
\DeclareMathOperator{\GL}{GL}

\DeclareMathOperator{\Spec}{Spec}

\renewcommand{\Spec}{\mathrm{Spec}}

\DeclareMathOperator{\CC}{\mathbb{C}}
\DeclareMathOperator{\barS}{\overline{S}}

\newcommand{\xdashrightarrow}[2][]{\ext@arrow 0359\rightarrowfill@@{#1}{#2}}

\newtheorem{theorem}[equation]{Theorem}
\newtheorem{lemma}[equation]{Lemma}

\theoremstyle{definition}
\newtheorem{defin}[equation]{Definition}

\newtheorem{example}[equation]{Example}

\makeatletter
\@addtoreset{equation}{section}
\makeatother

\title{Automorphisms of quasi-projective surfaces over fields of finite characteristic}
\author{Alexandra Kuznetsova}

\address{
National research university Higher School of economics, Russia, Usacheva str. 6, 119048;
\'Ecole Polytechnique, France, CMLS, Route de Saclay, 91128 Palaiseau.}

 \email{sasha.kuznetsova.57@gmail.com}

\begin{document}
\begin{abstract}
 We prove that the group of automorphisms of any quasi-projective surface $S$ in finite characteristic has the $p$-Jordan property.
\end{abstract}

\maketitle
\section{Introduction}
Groups of regular and birational automorphisms of algebraic varieties are very interesting to study. However, some of them
are immense, they can have infinite dimension or infinite number of connected components. One approach is to study the structure of their finite subgroups. 
J.-P. Serre in \cite{Serre_Cremona2} proved that any finite subgroup of the Cremona group of rank~$2$, i.e. the group of birational automorphisms of $\p^2$, has 
a normal abelian subgroup of bounded index. Then he conjectured that Cremona group of any rank satisfy the same property:
\begin{defin}
 A group $\Gamma$ satisfies the Jordan property if there exists a number $J$ 
 such that for any finite subgroup $G\subset \Gamma$ we can find a normal abelian subgroup $A\subset G$ of index less than or equal to~$J$.
\end{defin}
This property is named after C. Jordan who showed that it holds for $\GL_n(\CC)$,  see~\mbox{\cite[Theorem 36.13]{Jordan_thm}.}
Serre's conjecture motivated the study of the Jordan
property for regular and birational automorphism groups for many different varieties.

A foundational statement of this type was proved by V. Popov in \cite[Theorem 2.32]{Popov_def_Jordan}: he showed that 
in characteristic zero the group of birational automorphisms of any surface $S$ 
satisfies the Jordan property for all but one birational class of surfaces. Also S. Meng and D.-Q. Zhang in \cite{Meng-Zhang} showed that the Jordan property 
holds for groups of regular automorphisms of all projective varieties over a field of characteristic zero. 
Finally, Yu. Prokhorov and C. Shramov in \cite{PS_Jordan_property} proved the Jordan property for Cremona groups of all ranks assuming
the Borisov–Alexeev–Borisov conjecture which was later proved by C. Birkar in~\cite{Birkar_BAB}.

Another remarkable result in this area is due to T. Bandman and Yu. Zarhin \cite{Bandman-Zarhin_1}: they showed that 
groups of automorphisms of quasi-projective surfaces over a field of characteristic 0 have
the Jordan property.

An interesting question arises: can these results be extended to finite characteristic?
Many groups defined over fields of finite characteristic does not satisfy the Jordan property.
In fact, this property is not true even for the group~\mbox{$\GL_n(\overline{\mathbb{F}}_p)$.}
In view of this F. Hu suggested the following analogue of the Jordan property:
\begin{defin}[{\cite[Definition 1.6]{Fei_Hu}}]\label{def_p-Jordan}
 We say that the group $\Gamma$ is $p$-Jordan, if there exist constants~$J(\Gamma)$ and $e(\Gamma)$ depending only on $\Gamma$
 such that any finite subgroup $G\subset \Gamma$ contains a normal abelian subgroup~$A$ and 
 \begin{equation*}
  [G:A]\leqslant J(\Gamma)\cdot |G_p|^{e(\Gamma)},
 \end{equation*}
 where $G_p$ is a Sylow $p$-subgroup of $G$.
\end{defin}
This definition is a slight modification of the standard Jordan property and gives us some information about finite subgroups of the whole group
whose order is coprime with $p$. Definition \ref{def_p-Jordan} was motivated by a theorem proved by M. J. Larsen and R. Pink \cite{Larsen_Pink} 
which asserts  that the $p$-Jordan property holds for the group~\mbox{$\GL_n(\overline{\mathbb{F}}_p)$}, 
see also \cite{Brauer_Feit}.
Then F.Hu generalized the result by S. Meng and D.-Q. Zhang to finite characteristic:
\begin{theorem}[{\cite[Theorems 1.7, 1.10]{Fei_Hu}}]\label{thm_Hu}
 Let $k$ be a field of characteristic $p>0$.
 Then for any projective variety $X$ defined over~$k$ the group $\Aut(X)$ is $p$-Jordan.
\end{theorem}
The next interesting and natural question is whether we can prove the $p$-Jordan property for groups of birational automorphisms of 
projective varieties over a field of finite characteristic.
Recently, Y. Chen and C. Shramov studied this question and managed to generalize V. Popov's result \cite{Popov_def_Jordan} to finite characteristic:
\begin{theorem}[{\cite[Theorem 1.7]{Chen-Shramov}}]\label{thm_CS}
 Let $k$ be a field of characteristic $p>0$ and $S$ be an irreducible algebraic surface defined over $k$. 
 Then the group of birational automorphisms $\Bir(S)$ is $p$-Jordan unless $S$ is birational to the
 product $E\times \p^1$ of an elliptic curve $E$ and a projective line.
\end{theorem}

In this paper we are going to study another question connected to this theme, namely, we consider groups of regular automorphisms of quasi-projective
surfaces over a field of finite characteristic. Our goal is to prove an analogue of T. Bandman and Yu. Zarhin theorem \cite{Bandman-Zarhin_1} in finite characteristic:
\begin{theorem}\label{thm_aut_of_qp_surfaces}
 If $S$ is a quasi-projective surface defined over a field $k$ of characteristic $p>0$, then the group $\Aut(S)$ is $p$-Jordan.
\end{theorem}

The first part of the proof of this theorem is similar to that in \cite{Bandman-Zarhin_1}. It is a series of reductions leading us
to the situation when $S$ is a surface equipped with a surjective regular map $\pi\colon S\to E$ to an elliptic curve~$E$ such that
a general fiber of $\pi$ is $\p^1$. Then \cite{Bandman-Zarhin_1} shows that $\Aut(S)$ is a subgroup in $\Aut(\barS)$ for some smooth 
projective closure $\barS$ of $S$, then using the study \cite{Zarhin} of biregular automorphisms of projective ruled elliptic surfaces they get the result.

In the case of finite characteristic the analogue of the theorem proved in \cite{Zarhin} is unknown. 
In its place, we will use Theorem \ref{thm_Hu} in combination
with an analysis of the embedding $\Aut(S)\subset \Bir(\barS)$ and prove our theorem.

An important detail in our proof is the fact that if $\barS\setminus S$ contains 
a curve $C$ which maps dominantly to $E$ then any birational automorphism of $\barS$ does not contract $C$, see Lemma \ref{lemma_horizontal_curves_in_complement}. 
This implies that any regular
automorphism of~$S$ can be extended to a regular automorphism of a larger open subset of $\barS$ which contains an open subset of $C$.
This is true since the multisection $C$ cannot be a rational curve.

This is the reason why the proof of Theorem \ref{thm_aut_of_qp_surfaces} cannot be easily extended to dimension $3$
since in finite characteristic there exist several examples of unirational non-rational surfaces $A$, in Example \ref{ex_Fermat} we recall a construction of 
such surface introduced by T. Shioda. If the direct
product~\mbox{$\p^1\times A$} contains a rational surface which projection to $A$ is surjective, then Lemma \ref{lemma_horizontal_curves_in_complement}
fails for open subsets of the product $\p^1\times A$ and possibly this can lead to a construction of quasi-projective threefold which is not $p$-Jordan.

\medskip

{\bf Acknowledgements.} I am very grateful to my advisor, Constantin Shramov, for suggesting this
problem as well as for his patience and invaluable support.  This work is supported by Russian Science Foundation
under grant \textnumero 18-11-00121.

\section{Automorphisms of quasi-projective surfaces}
In this section we are going to prove Theorem \ref{thm_aut_of_qp_surfaces}. 
We will need the following well-known assertion which shows that the $p$-Jordan property holds for finite extensions of $p$-Jordan groups:
\begin{lemma}[{see, e.g., \cite[Lemma 2.8]{Chen-Shramov}}]\label{lemma_Jordan_subgroup}
 Assume that a group $\Gamma$ contains a $p$-Jordan subgroup $\Gamma'$ of finite index.
 Then $\Gamma$ is $p$-Jordan.
\end{lemma}

We consider an irreducible quasi-projective surface $S$. Denote by $\barS$ its projective closure. 
That is, $\barS$ is a projective surface such that $S$ embeds into $\barS$ as a dense subvariety.
Any automorphism of $S$ induces a birational automorphism of~$\barS$.
\begin{lemma}\label{lemma_except_P1xE}
 If $S$ is an irreducible surface which is not birational to a product  $\p^1\times E$ of a projective line and an elliptic curve~$E$,  then the group $\Aut(S)$ is $p$-Jordan.
\end{lemma}
\begin{proof}
 Any automorphism of $S$ induces a birational automorphism of $\barS$. 
 Since the group $\Bir(\barS)$ is $p$-Jordan by Theorem \ref{thm_CS} we get the result.
\end{proof}
Now we study quasi-projective surfaces $S$ which are birational to the product $\p^1\times E$ of a projective line and an elliptic curve~$E$. 
Denote by $\pi\colon S\to E$ the projection to $E$. It is the Albanese map; thus it is regular.
Finite subgroups of automorphisms of such surfaces can be represented as extensions of two groups.
\begin{lemma}\label{lemma_exact_sequence}
 Assume that $S$ is a smooth quasi-projective surface over a field $k$ birational to $\p^1\times E$, here~$E$ is a smooth curve of positive genus
 and $\pi\colon S\to E$ is a projection. Then there exists an exact sequence:
 \begin{equation}\label{eq_exact_sequence}
  1\to \Bir(S)_{\pi} \to \Bir(S) \to \Gamma,
 \end{equation}
 where $\Bir(S)_{\pi}\subset \PGL(2, k(C))$ and $\Gamma\subset \Aut(E)$. Moreover, the group $\Bir(S)_{\pi}$ is $p$-Jordan.
\end{lemma}
\begin{proof}
 See, for instance, \cite[Corollary 2.14, Lemma 4.16]{Chen-Shramov}.
\end{proof}

Now we show that in several cases the group of automorphisms of a rationally connected fibration over an elliptic curve
happens to be $p$-Jordan. Here we consider the case when $\Aut(S)$ fixes a finite subset of fibers of $\pi$.
\begin{lemma}\label{lemma_fixed_Z_on_E}
 Assume that $S$ is a smooth quasi-projective  surface birational to $\p^1\times E$ and~$\pi\colon S\to E$ is the projection.
 Denote by $\Gamma$ the group of automorphisms of $E$ induced by $\Aut(S)$.
 If there exists a finite subset $Z\subset E$ such that $\Gamma$ preserves $Z$, then the group $\Aut(S)$ is $p$-Jordan.
\end{lemma}
\begin{proof}
 Since $\Gamma$ preserves $Z$ there exists a subgroup $H$ of $\Gamma$ of index $|Z|$ which stabilizes a point $z\in Z$.
 The group~$H$ acts faithfully on $E$; thus, it is finite and $|H|\leqslant 24$ by \cite[Exercise A.1(b)]{Silverman}. 
 Therefore,~\mbox{$|\Gamma|\leqslant 24|Z|$.}
 
 Thus, by Lemma \ref{lemma_exact_sequence} The group $\Aut(S)$ contains a $p$-Jordan subgroup $\Aut(S)_{\pi}$ of finite index.
 Therefore,~$\Aut(S)$ is also $p$-Jordan by Lemma \ref{lemma_Jordan_subgroup}.
\end{proof}

Starting with a smooth quasi-projective surface $S$ birational to $\p^1\times E$ we can choose a good projective closure of $S$.

\begin{lemma}\label{lemma_smooth_projective_closure}
 Assume that $S$ is a smooth quasi-projective surface birational to $\p^1\times E$ over an algebraically closed field $k$, 
 where $E$ is an elliptic curve and $\pi\colon S\to E$ is a projection.
 Then there exists a smooth projective variety $\barS$ and a open embedding $\iota\colon S\hookrightarrow \barS$ such that 
 $\pi$ induces a morphism $\overline{\pi}$ from $\barS$ to $E$:
 \begin{equation*}
  \xymatrix{
  S \ar@{^{(}->}[r]^{\iota} \ar[d]_{\pi} & \barS \ar[ld]^{\overline{\pi}} \\ E
  }
 \end{equation*}
\end{lemma}
\begin{proof}
 Consider some projective closure $\barS'$ of $S$.
 We normalize $\barS'$ and then after a sequence of blow-ups we obtain a smooth model $\barS$.
 Since $S$ is a smooth dense subset in $\barS'$ it does not intersect centers of blow-ups. Thus, $\barS$ is a projective closure of $S$.
 
 The induced 
 map $\overline{\pi}\colon \barS\to E$ is regular 
 since $\overline{\pi}$ is the Albanese map.
\end{proof}
The complement $\barS\setminus S$ can contain multisections of $\overline{\pi}$. However, any automorphism of $S$ cannot contract such multisection.

\begin{lemma}\label{lemma_horizontal_curves_in_complement}
Assume that $S$ is a smooth quasi-projective surface birational to $\p^1\times E$ over an algebraically closed field $k$, 
where $E$ is a smooth projective curve of positive genus 
and $\pi\colon S\to E$ is a projection.
If $C$ is a curve in $S$ such that ${\pi}(C)$ is a dense subset in $E$, then for any birational automorphism $g$ of $S$ the curve~$C$ does not lie in the exceptional locus of $g$.
\end{lemma}
\begin{proof}
 Let $C\subset S$ be an irreducible curve such that ${\pi}(C) = E$.
 Then $C$ dominates a curve of positive genus and by \cite[Corollary IV.2.4]{Hartshorne} and \cite[Proposition IV.2.5]{Hartshorne} the curve $C$ cannot be rational.
 Thus, $C$ is not contractible.
\end{proof}

There is many different ways to choose a smooth projective closure of $S$. Some of them are very useful:

\begin{defin}
 Assume that $\barS$ is a smooth quasi-projective surface birational to $\p^1\times E$  over an algebraically closed field $k$, 
 where $E$ is an elliptic curve and $\pi\colon \barS\to E$ is a projection.
 Let $\barS$ be a smooth projective closure of $S$. 
 We say that it is \emph{minimal} if there is no $(-1)$-curve $C$ in the complement $\barS\setminus S$ lying in fiber of~$\overline{\pi}$.
\end{defin}
By the following lemma any smooth ruled quasi-projective surface $S$ admits a minimal projective closure.
\begin{lemma}\label{lemma_minimal_closure}
  Assume that $S$ is a smooth quasi-projective surface birational to $\p^1\times E$  over an algebraically closed field $k$, where $E$ is an elliptic curve and $\pi\colon S\to E$ is a projection.
  Then there exists a minimal projective closure of $S$. Moreover, if all fibers of $S$ are smooth, then the same is true for $\barS$.
\end{lemma}
\begin{proof}
 In \cite[Section 4]{Fujita} the same assertion was proved in characteristic 0. In fact all arguments from that paper work in positive characteristic.

 By Lemma \ref{lemma_smooth_projective_closure} we consider some smooth projective closure $\overline{\pi}\colon \barS\to E$ of $S$.
 If the complement $\barS\setminus S$ contains a $(-1)$-curve lying in fiber of $\overline{\pi}$, then it can be blown down. 
 Thus, we obtain a new smooth projective closure $\overline{\pi}\colon\barS'\to E$ of $S$ and the number of codimension 1 
 components of $\barS'\setminus S$ is less then one of~$\barS\setminus S$. Repeating this process several times we can construct
 a smooth projective closure $\overline{\pi}_0\colon\barS_0\to E$ of~$S$ such that $\barS_0\setminus S$ does not contain $(-1)$-curves 
 lying in fibers of $\overline{\pi}_0$.
 
 If some fiber $F$ of $\overline{\pi}_0$ is singular, then $F$ is reducible, its components are rational curves and the dual graph is a tree. 
 In particular it contains two $(-1)$-curves corresponding to hanging verteces of the dual graph of $F$. If we assume that all fibers of $S$ are 
 smooth irreducible and reduced this implies that $\barS_0\setminus S$ contains a $(-1)$-curve lying in the fiber of $\overline{\pi}_0$. 
 Thus, all fibers of $\overline{\pi}_0$ are smooth.
\end{proof}

Now we are ready to prove Theorem \ref{thm_aut_of_qp_surfaces}.

\begin{proof}[Proof of Theorem \textup{\ref{thm_aut_of_qp_surfaces}}]
We can assume that $k$ is algebraically closed.
If the quasi-projective surface $S$ is not irreducible, we consider its irreducible component $S_0$. By Lemma \ref{lemma_Jordan_subgroup} the group $\Aut(S)$ 
is $p$-Jordan if $\Aut(S_i)$ is $p$-Jordan for all irreducible components~\mbox{$S_1,\dots, S_m$} of $S$. Thus, from now on we assume that $S$ is irreducible.

By Lemma \ref{lemma_except_P1xE} we reduce to the case when $S$ is birational to $\p^1\times E$ since otherwise $\Aut(S)$ is a $p$-Jordan group.

Any regular automorphism of $S$ can be lifted to its normalization; thus, we can assume that $S$ is normal.
Replacing $S$ by its minimal resolution of singularities we can assume that it is smooth. The group of automorphisms of $S$ embeds to the group of automorphism
of the resolution. Thus, we can assume that $S$ is smooth.

Since $S$ is birational to $\p^1\times E$ we can consider the projection $\pi\colon S \to E$. 
It is regular since it is the Albanese map from a smooth variety.

If $\pi(S)$ is a proper subset  or if there are non-reduced or reducible fibers on $S$, then by Lemma \ref{lemma_fixed_Z_on_E} we get that $\Aut(S)$ is $p$-Jordan. 
From now on we assume that all fibers of $S$ are smooth and $\pi(S) = E$.

By Lemma \ref{lemma_minimal_closure} we can construct a minimal closure $\barS$ of
$S$ such that all fibers of $\barS$ are smooth and irreducible. 
Since $\barS$ is minimal and $\pi(S) = E$, then the complement $\barS\setminus S$ is a set of curves which do not lie in fibers of $\overline{\pi}$.
If $g\in\Aut(S)$, then it induces a birational automorphism $\overline{g}\in\Bir(\barS)$. The exceptional locus of~$\overline{g}$ lies in $\barS\setminus S$.
Thus, by Lemma \ref{lemma_horizontal_curves_in_complement} we get that the exceptional locus of $\overline{g}$ is empty.
Then $g$ induces a regular automorphism of $\barS$.

We get that $\Aut(S)$ is embedded to $\Aut(\barS)$ for a minimal projective closure $\barS$ of $S$. By Theorem \ref{thm_Hu} we get that $\Aut(S)$ is $p$-Jordan.
\end{proof}

 The proof of Theorem \ref{thm_aut_of_qp_surfaces} can not be easily modified to the case of quasi-projective threefolds. 
 The main problem is that Lemma~\ref{lemma_horizontal_curves_in_complement} may fail for $\p^1$-fibration over
 non-rational surfaces, since there are examples of non-rational unirational surfaces in finite characteristic. 
 Here we recall a construction of such surface.

  \begin{example}\label{ex_Fermat}
  We recall a construction of a unirational general type surface discovered by T. Shioda~\cite{Shioda_1}.
  The surface $A$ is a Fermat surface lying in $\p^3$ over an algebraically closed field $k$ of characteristic $p$:
  \begin{equation*}
   A = \left\{\left. (x_0:x_1:x_2:x_3)\in\p^3\ \right|\ x_0^n + x_1^n + x_2^n +x_3^n = 0 \right\}
  \end{equation*}
  We choose $n\geqslant 5$ to be such that there exists $\nu$ and $p^{\nu}\equiv -1 \mod{n}$.
  Then we consider the following finite cover:
  \begin{equation*}
   f\colon \Spec\left(k[x_1,x_2,x_3,(x_1+x_2)^{\frac{1}{p}}]/( x_1^n + x_2^n +x_3^n +1)\right) \to \Spec\left(k[x_1,x_2,x_3]/( x_1^n + x_2^n +x_3^n +1)\right) \subset A.
  \end{equation*}
 The affine variety on the left hand side is rational by \cite[Proposition 1]{Shioda_1}, we denote its projective closure by $T$. 
 The right hand side is an open subset of $A$. 
 
 This allows us to embed the rational surface $T = \Spec\left(k[x_1,x_2,x_3,(x_1+x_2)^{\frac{1}{p}}]/( x_1^n + x_2^n +x_3^n +1)\right)$ to 
 the product $X = \p^1\times A$. 
 Note that $\Spec\left(k[x_1,x_2,x_3,y]/( x_1^n + x_2^n +x_3^n +1)\right)$ is an affine chart in $X$.
 Then the embedding $\iota\colon T\to X$ is defined by the homomorphism of algebras
 \begin{equation*}
  \iota^*\colon k[x_1,x_2,x_3,y]/( x_1^n + x_2^n +x_3^n +1) \to k[x_1,x_2,x_3,(x_1+x_2)^{\frac{1}{p}}]/( x_1^n + x_2^n +x_3^n +1),
  \end{equation*}
  which is defined by the following formula:
  \begin{equation*}
  \iota^*(y) = (x_1+x_2)^{\frac{1}{p}}.
 \end{equation*}
 \end{example}
  
 Other examples of unirational non-rational surfaces in finite characteristic can be found in these 
 papers:~\cite{Katsura}, \cite{Katsura_Schutt}, \cite{Shioda_Katsura}, \cite{Miyanishi_1}, \cite{Miyanishi_2}, \cite{Ohhira}, \cite{Rudakov_Shafarevich}, 
 \cite{Shioda_2}, \cite{Shioda_3}.

\bibliographystyle{alpha}
\bibliography{Jordan}
 
\end{document}